\documentclass[11pt]{amsart}
\usepackage{mathrsfs}
\usepackage{mathtools}
\usepackage{dsfont}
\usepackage{amsmath,amssymb,amsthm,upref,graphicx,mathrsfs}
\usepackage{enumerate}
\usepackage{bbm}

\usepackage{color}
\usepackage[
  colorlinks=true,
  linkcolor=blue,
  citecolor=blue,
  urlcolor=blue]{hyperref}

\usepackage{appendix}

\numberwithin{equation}{section}

\usepackage{geometry} 

\numberwithin{equation}{section}

\theoremstyle{plain}
\newtheorem{thm}{Theorem}[section]
\newtheorem{lem}[thm]{Lemma}
\newtheorem{prop}[thm]{Proposition}
\newtheorem{cor}[thm]{Corollary}
\newtheorem{defn}[thm]{Definition}
\newtheorem{rem}[thm]{Remark}
\newtheorem{com}[thm]{Comment}

\newcommand{\R}{\mathbb{R}}
\newcommand{\N}{\mathbb{N}}

\newcommand{\Z}{\mathbb{Z}}

\newcommand{\E}{\mathbb{E}}

\newcommand{\V}{\mathbb{V}}

\newcommand{\Ex}{\mathcal{E}}

\newcommand{\normmm}[1]{{\left\vert\kern-0.25ex\left\vert\kern-0.25ex\left\vert #1
   \right\vert\kern-0.25ex\right\vert\kern-0.25ex\right\vert}}

\newcommand{\T}{\tau}

\newcommand{\M}{\mathcal{M}}
\newcommand{\A}{\mathcal{A}}

\begin{document}

\title[Deviation inequalities]{Large deviation inequalities for noncommutative martingales}
 
\author[Jiao]{Yong Jiao}
\address{School of Mathematics and Statistics,  HNP-LAMA, Central South University, Changsha 410083, China}
\email{jiaoyong@csu.edu.cn}

\author[Luo]{Sijie Luo${}^{*}$}
\address{School of Mathematics and Statistics, HNP-LAMA, Central South University, Changsha 410083, China}
\email{sijieluo@csu.edu.cn}

\author[Zhou]{Dejian Zhou}
\address{School of Mathematics and Statistics, HNP-LAMA,
Central South University, Changsha 410083, China}
\email{zhoudejian@csu.edu.cn}
%

\thanks{$*$ Corresponding author.}



%
%
\begin{abstract}
We establish noncommutative analogs of some well-known large deviation inequalities for noncommutative random variables. Firstly, for the noncommutative independent case,  we characterize the uniformly exponential integrability of random variables in terms of large deviation inequalities. Secondly, for noncommutative martingale differences, we establish two deviation inequalities according to the exponential integrability and $L_{p}$-boundedness of the martingale differences, respectively. Finally, we establish a noncommutative version of Gordin’s decomposition, which enables us to derive a noncommutative ergodic theorem via deviation inequalities for noncommutative martingales.
\end{abstract}

\maketitle

\section{Introduction}
The study of large deviation inequalities is one of the essential themes in probability theory, which has been well developed by Bernstein, Cram\'{e}r,  Hoeffding \cite{Ho1963}, Azuma \cite{Az1967}, \cite{dlP1999} and many other mathematicians. We refer to the comprehensive monographs \cite{CT1997} and \cite{Pe1995} for more details on the active theory. Let $(d_{j})_{j=1}^{\infty}$ be a sequence of random variables on a fixed probability space $(\Omega,\mathcal{F},\mathbb{P})$ and $S_{n}=\sum_{j=1}^{n}d_{j}$ be the partial sum of $(d_{j})_{j=1}^{\infty}$. In the theory of deviation inequality, one mainly focus on inequality as the following form:
\begin{equation}\label{eq:ldi}
	\mathbb{P}(|S_n|>nr)= O(e^{-c_rn}),~\mbox{for}~r>0,
\end{equation}
and $c_{r}$ is a positive constant depending only on $r$. When $(d_{j})_{j=1}^{\infty}$ is a mean zero independent and identically distributed (i.i.d. for short) sequence, a well-known result states that \eqref{eq:ldi} holds if and only if the sequence fulfills the Cram\'{e}r condition (see e.g. \cite[p.\,137]{Pe1995}): there is $\delta>0$ with
\[
\sup_{j\in \N}\E[e^{\delta|d_{j}|}]<\infty.
\]
If the sequence $(d_{j})_{j=1}^{\infty}$ fails to be i.i.d., Lesigne  and Voln\'{y} \cite{LV2001} proved that, under the Cram\'{e}r condition, one gets
\begin{equation}\label{eq:ldi md}
	\mathbb{P}(|S_n|>nr)= O(e^{-c_rn^{1/3}}),
\end{equation}
whenever $(d_{j})_{j=1}^{\infty}$ forms a martingale differences. Notably, Lesigne  and Voln\'{y} \cite{LV2001} showed that the power $1/3$ in \eqref{eq:ldi md} is optimal by constructing a stationary and ergodic martingale differences such that $\mathbb{P}(|S_n|>n)\geq e^{-cn^{1/3}}$. 
On the other hand, if the martingale differences satisfy the $L_{p}$-boundedness condition as $\sup_{j\in \N}\|d_{j}\|_{p}<\infty$ for some $1<p<\infty$, we have 
\begin{equation}\label{eq:ldi md p}
\mathbb{P}(|S_n|>nr)= O\left(c_{r,p}n^{p\left(\frac{1}{\min\{2,p\}}-1\right)}\right).
\end{equation}
Lesigne and Voln\'{y} \cite[Theorem 3.6]{LV2001} obtained \eqref{eq:ldi md p} for $2\leq p<\infty$, and demonstrated its optimality for strictly stationary and ergodic martingale differences. Subsequently,  Li \cite{Li2003} established \eqref{eq:ldi md p} for the case $1<p<2$, and obtained the optimality for i.i.d. sequences. In particular, \eqref{eq:ldi md p} serves as an efficient tool in deriving the convergence speed in the ergodic theory when combined with Gordin's decomposition. We refer readers to \cite[Corollary 4.4]{LV2001} for more details. Recently, in the series of works of Fan, Grama, and Liu \cite{FGL2012, FGL2012-2, FGL2013, FGL2015, FGL2017}, the authors improved upon Lesigne and Voln\'{y}'s result by establishing various large deviation inequalities and exploring their applications. Specifically, the authors in \cite{FGL2012} extended  \eqref{eq:ldi md} to the following
\begin{equation}\label{eq:ldi mcd}
\mathbb{P}(|S_n|>nr)= O\left(e^{-c_rn^{\alpha}}\right),\quad \alpha\in (0,1),
\end{equation}
whenever the sequence satisfies a modified Cram\'{e}r condition 
\[
\sup_{j\in \N}\E[\exp\{|d_{j}|^{2\alpha/(1-\alpha)}\}]<\infty,~\mbox{for some}~\alpha\in(0,1).
\]
Optimality of \eqref{eq:ldi mcd} was also obtained in \cite{FGL2012}, and it is easy to see that \eqref{eq:ldi mcd} reduces to \eqref{eq:ldi md} whenever $\alpha=1/3$.

Inspired by the study of noncommutative probability, we aim to extend \eqref{eq:ldi md}-\eqref{eq:ldi mcd} to the noncommutative framework in the present paper. In the fundamental work of Pisier and Xu \cite{PX1997}, they found the appropriate definition of martingale Hardy spaces and established the noncommutative Burkholder-Gundy inequality. This breakthrough led to the development of the noncommutative martingale theory, which has become a critical field in probability and operator theory, attracting considerable attention and expanding rapidly. We refer the reader to Junge \cite{Ju2002} for noncommutative Doob maximal inequality; to   Randrianantoanina \cite{Ra2002} for  the weak type $(1,1)$ inequality for martingale transform; to \cite{JX2003, JX2008} for  noncommutative Burkholder/Rosenthal inequalities; to Parcet and Randrianantoanina  \cite{PR2006} for the Gundy decomposition of noncommutative martingales. Up to this point, the noncommutative martingale theory has been applied to investigate various areas such as random matrices, Banach space theory, and harmonic analysis (see e.g. \cite{JX2008}, \cite{Pa2009}, \cite{SZ2020}). We also refer the reader to \cite{CRX2023, JZZ2023, Ra2023} for very recent progress of noncommutative martingales.

As mentioned above, after Pisier and Xu's fundamental work, a great deal of effort has been expended on moment inequalities for noncommutative martingales, providing a comprehensive understanding of the theory. However, to the best of our knowledge, extending deviation inequalities to the noncommutative framework has not yet to be extensively explored. Currently, there are only a few established results in this area, including noncommutative Bernstein inequality and Bennett inequality established by Junge and Zeng in \cite{{JZ2013},{JZ2015}}, and the noncommutative Hoeffding-Azuma inequality and McDiarmid inequality by Sadeghi and Moslehian in \cite{SM2014}. One of the primary motivations of this paper is to extend deviation inequalities to the noncommutative setting. More precisely, we apply the fundamental tools from noncommutative martingale theory to establish the noncommutative analogs of \eqref{eq:ldi md}-\eqref{eq:ldi mcd}.

Our paper is organized as follows. In Section \ref{sec:pre}, we recall background and necessary concepts in the noncommutative martingale theory. Moreover, a slight strengthening of the noncommutative Hoeffding-Azuma inequality via Cuculescu projections is also included. Our main result concerns deviation inequalities for sum of independent noncommutative random variables are presented in Section \ref{sec:ldi}. Specifically, Theorem \ref{thm:ldi} is a noncommutative counterpart of \eqref{eq:ldi}. Section \ref{sec:ldi md} is devoted to extending and refining two basic deviation inequalities of Lesigne and Voln\'{y} to noncommutative martingales. Finally, in Section \ref{sec:application}, we adapt the decomposition initiated by Voln\'{y} \cite{Vo1993} and apply our Theorem \ref{Lp-version 1} to study the convergence speed in noncommutative ergodic theory.

Throughout the paper, all notations and symbols are standard. Let $(\M,\tau)$ be a fixed von Neumann algebra equipped with a normal faithful tracial state $\tau$ with the unit $\mathbf{1}$, and we simply refer $(\M,\tau)$ as a noncommutative probability space. For a number $p$, the notation $c_{p}$ means a constant depending only on $p$ and we use $\asymp_{p}$ to stand for the equivalence up to some constant $c_{p}$, that is, $A\asymp_{p} B$ if and only if there exist  $c_{p}$ and $d_{p}$ such that $c_{p}A\leq B\leq d_{p}A$. For positive functions $F$ and $G$ on $\N$, we use the big $O$ notation $F(n)=O(G(n))$ to stand that there exists a universal constant $K>0$ with $F(n)\leq KG(n)$ as $n\to \infty$.



\section{Preliminaries}\label{sec:pre}

\subsection{Noncommutative Lebesgue spaces}
The algebra of all $\tau$-measurable operators are denoted by  $L_0(\mathcal{M})$.    Suppose that $a$ is a self-adjoint $\tau$-measurable operator and let $a=\int_{-\infty}^{\infty} \lambda d e_{\lambda}$ stand for its spectral decomposition. For any Borel subset $B$ of $\mathbb{R}$, the spectral projection of $a$ corresponding to the set $B$ is defined by $\mathds{1}_B(a)=\int_{-\infty}^{\infty} \mathds{1}_B(\lambda) d e_{\lambda}$.  
For $1\leq p\leq \infty$, let $L_p(\mathcal{M}, \tau)$ (simply $L_p(\mathcal{M})$) be the associated noncommutative Lebesgue spaces.  As usual, $L_{\infty}(\mathcal{M})$ is just $\mathcal{M}$ with the usual operator norm $\|\cdot\|_{\M}$. For $1\leq p<\infty$, the norm on $L_p(\M)$ is naturally defined by 
$$\|x\|_p=[\tau(|x|^p)]^{1/p},\quad  x\in L_p(\M),$$
where $|x|=(x^*x)^{1/2}$ is the usual modulus of $x$. 

For $x\in L_0({\M})$, the generalized singular value function $\mu(x)$ is defined by
\[
\mu(t,x)=\inf\{s>0:\tau\big(\mathds{1}_{(s,\infty)}(|x|)\big)\leq t\},\quad t>0.
\]
The function $t\mapsto \mu(t,x)$ is decreasing and right-continuous; for a more detailed study of the singular value function we refer the reader to \cite{FK1986}. According to \cite{FK1986}, the Fubini theorem yields that, for each $1\leq p<\infty$, we have
\begin{equation}\label{eq:d-r}
\|x\|_{p}^p= \int_0^1\mu(t,x)^pdt= p\int_0^{\infty}\lambda^{p-1} \tau[\mathds{1}_{(\lambda,\infty)}(|x|)] d\lambda.
\end{equation}
For $0<p<\infty$ and  positive element $a\in L_{p}(\M)$, by the functional calculus of $a$ and the Lebesgue-Stieltjes measure associated with function $F_{a}(t)=\tau\left(\mathds{1}_{(t,\infty)}(a)\right)$, we have
\begin{equation}\label{eq:dsb}
\tau(a^p\mathds{1}_{(u,\infty)}(a))= -\int_u^{\infty}t^p dF_a(t).
\end{equation}

\subsection{Noncommutative martingales}

Let $(\mathcal{M}_n)_{n\geq 1}$ be an increasing sequence of von Neumann subalgebras of  $\mathcal{M}$ such that $\bigcup_{n\geq1}\mathcal{M}_n$ is weak-$*$ dense in $\mathcal{M}.$
Let $\mathcal{E}_n$ be the  conditional expectation (the existence of $\mathcal E_n$ is referred to \cite[Proposition V.2.36]{Ta1979}) from $\mathcal{M}$ onto $\mathcal{M}_n.$ An adapted sequence $x=(x_n)_{n\geq1}$ in $L_{1}(\M)$ is called a noncommutative martingale with respect to $(\mathcal M_n)_{n\geq1}$ if
\[
\mathcal E_n(x_{n+1})=x_n,\qquad \forall n\geq1.
\]
The corresponding martingale differences $(dx_{k})_{k\geq 1}$ for a given sequence $x=(x_n)_{n\geq1}$ is defined by  $dx_{1}\coloneqq x_{1}$ and
\[
dx_{k}\coloneqq x_k-x_{k-1},\qquad \forall k\geq2.
\]
In the sequel, we will remove the term "noncommutative" when referring to a noncommutative martingale unless it causes confusion.
A martingale $(x_n)_{n\geq1}\subseteq L_p(\mathcal{M})$ for some $1\leq p\leq \infty$ is called an $L_p$-bounded martingale if
\[
\|x\|_{p}\coloneqq \sup_{n\geq 1}\|x_{n}\|_{p}<\infty.
\]

The following noncommutative Burkholder-Gundy inequality, due to Pisier and Xu \cite{PX1997}, is one of significant tools in the noncommutative martingale theory.
\begin{thm}\label{Nc-BG}
Let $(x_{n})_{n\geq 1}$ be a noncommutative martingale. Then, for each $n\in \mathbb{N}$, we have
\begin{equation*}
 \|x_{n}\|_{L_{p}(\M)}\asymp_{p} \inf_{dx_k=dy_k+dz_k}\left\{\left\|\left(\sum_{k=1}^n|dy_k|^2\right)^{1/2}\right\|_p+\left\|\left(\sum_{k=1}^n|dz_k^*|^2\right)^{1/2}\right\|_p\right \}
\end{equation*}
for $1<p<2$, and 
\begin{equation*}
 \|x_{n}\|_{L_{p}(\M)}\asymp_{p} \max\left\{\left\|\left(\sum_{k=1}^n|dx_k|^2\right)^{1/2}\right\|_{p},\left\|\left(\sum_{k=1}^n|dx_k^*|^2\right)^{1/2}\right\|_{p}\right\},\quad 2\leq p<\infty.
\end{equation*}
\end{thm}

\subsection{Noncommutative independence}
Elements in $L_0(\mathcal M)$ are called (noncommutative) random variables, and we say $x\in L_1(\mathcal M)$ is mean zero if $\tau(x)=0$.
Following \cite[Page 233]{JX2008},  we recall the noncommutative independence as follows.
\begin{defn}\label{def-ind}
	Let $(\mathcal M,\tau)$ be a noncommutative probability space. Assume that $\mathcal{N}$ and $(\mathcal{M}_k)_{k\geq0}$ are subalgebras of $\mathcal{M}$ such that $\mathcal{N}\subseteq \mathcal{M}_k$ for each $k.$ We further assume that there exist trace preserving normal conditional expectations $\mathcal{E}_{\mathcal{N}}:\M\to \mathcal{N}$ and $\mathcal{E}_{\mathcal{M}_{k}}:\M\to \mathcal{M}_{k}$ for all $k\geq0$.
	\begin{enumerate}[(i)]
		\item We say that a sequence $(\mathcal M_k)_{k\geq0}$ of von Neumann subalgebras in $\mathcal M$ are independent  with respect to $\mathcal{E}_{\mathcal{N}}$ (the conditional expectation from $\mathcal{M}$ to $\mathcal{N}$), $\mathcal{E}_{\mathcal{N}}(xy)=\mathcal{E}_{\mathcal{N}}(x)\mathcal{E}_{\mathcal{N}}(y)$ holds for every all $x\in \mathcal M_k$  and for every $y$ in the von Neumann algebra generated by $(\mathcal M_j)_{j\neq k}$.
		\item A sequence $(x_k)_{k\geq0}\subseteq L_0(\mathcal M)$ is said to be independent with respect to $\mathcal{E}_{\mathcal{N}}$, if the unital von Neumann subalgebras $\mathcal M_k$, $k\geq0$, generated by $x_k$ are independent with respect to $\mathcal{E}_{\mathcal{N}}$.
		\item A sequence $(x_k)_{k\geq0}\subseteq L_0(\mathcal M)$ is said to be independent, if it is independent  with respect to $\tau$.
	\end{enumerate}
\end{defn}

\begin{rem}\label{re1}
Let $(x_k)_{k\geq0}\subseteq L_{1}(\M)$ be a sequence, which is  independent with respect to $\mathcal{E}_{\mathcal{N}}$ such that $\mathcal{E}_{\mathcal{N}}(x_k)=0$ for each $k\geq0$. For each $k\geq 0$, denotes by $\mathcal{A}_k:=vN(x_{0},\dots,x_{k})$ the von Neumann subalgebras generated by $(x_{j})_{j=0}^{k}$. Then \cite[Lemma 1.2 and Remark 1.1]{JX2008} yields that
\[
\mathcal{E}_{\mathcal{A}_{k-1}}(x_k)=\mathcal{E}_{\mathcal{N}}(x_k)=0,\quad k\geq0,
\]
where  $\mathcal{E}_{\mathcal{A}_{k-1}}$ is the conditional expectation from $\mathcal{M}$ onto $\mathcal{A}_{k-1}$. Hence, $(x_k)_{k\geq0}$ forms a martingale differences with respect to the filtration $(\mathcal{A}_k)_{k\geq0}$.
\end{rem}

\subsection{Noncommutative Azuma inequalities }

The following Azuma inequality for noncommutative martingale is due to Sadeghi and Moslehian \cite{SM2014}, which serves as one of the basic tools in our study of deviation inequalities.

\begin{thm}[Sadeghi-Moslehian]\label{Nc-Azuma}
Let $(x_{j})_{j=1}^{n}$ be a self-adjoint martingale such that $\|dx_{j}\|_{\infty}\leq c_{j}$ for each $j$. Then we have
\[
\tau(\mathds{1}_{(r, \infty)}(|x_{n}|))\leq 2\exp\left\{\frac{-r^{2}}{2\sum_{j=1}^{n}c_j^2}\right\},~\mbox{for}~r>0,~n\in \N.
\]
\end{thm}

Combining the Cuculescu projections, we strengthen the noncommutative Azuma inequality as follows, which is of independent interest.

\begin{prop}\label{Nc-Azuma-max}
Suppose that $(x_{j})_{j=1}^{n}$ is a self-adjoint martingale such that $\|dx_{j}\|_{\M}\leq c_{j}$ for each $j$. Then, for each $\lambda>0$, there exist projections $q_N^{(\lambda)}$ satisfying
\[
\sup_{1\leq k\leq N}\|q_N^{(\lambda)} x_kq_{N}^{(\lambda)}\|_{\M}\leq \lambda,
\]
and
\[
\tau(1-q_N^{(\lambda)})\leq 2\exp\left\{\frac{-\lambda^2}{D\sum_{j=1}^Nc_j^2}\right\},
\]
for some universal constant $D>0$.
\end{prop}

To prove this theorem, we first recall  the so-called Cuculescu projections associated to a given  martingale $(x_{j})_{j\geq 1}$ as follows. For a given self-adjoint martingale $(x_{j})_{j\geq 1}$, set $q_0^{(\lambda)}={\bf 1}$ and define inductively that
\[
q_n^{(\lambda)} := q_{n-1}^{(\lambda)}\mathds{1}_{[-\lambda,\lambda]} \Big(
q_{n-1}^{(\lambda)} x_{n} q_{n-1}^{(\lambda)} \Big).
\]

\begin{prop}[{\cite[Proposition~1.4]{PR2006}}]\label{Cuculescu}
For $\lambda\in \R$, let $\{q^{(\lambda)}_{n}\}_{n\geq 1}$ be the Cuculescu projections which satisfies the following properties: For each $n\geq 1$, we have that
\begin{enumerate} [{\rm (i)}]
\item  $q_{n}^{(\lambda)} \in \M_{n}$ and $(q_{n}^{(\lambda)})_{n\geq1}$ is decreasing;
\item  $q_{n}^{(\lambda)}$ commutes with $q_{n-1}^{(\lambda)} x_{n} q_{n-1}^{(\lambda)}$;
\item  $|q_{n}^{(\lambda)} x_{n}q_{n}^{(\lambda)}| \leq \lambda q_{n}^{(\lambda)}$;
\item
\[
\tau\left( {\mathbf{1}}-q^{(\lambda)}_{n}\right)\leq \frac{1}{\lambda}\tau\left(\left(\mathbf{1}-q^{(\lambda)}_{n}\right)|x_{n}|\right).
\]
\end{enumerate}
\end{prop}

\begin{lem}\label{lem-e1}
For $1\leq p<\infty$, $\lambda>0$, let $(q^{(\lambda)}_{n})_{n\geq 1}$ be the Cuculescu projections associated with the $L_{p}$ self-adjoint martingale  $x=(x_n)_{n\geq1}$. Then we have
\[
\lambda \left(\tau\left(\mathbf{1}-q^{(\lambda)}_{N}\right)\right)^{1/p} \leq  \|x_{N}\|_{p},~\mbox{for}~N\in \N.
\]
\end{lem}

\begin{proof} 
The case $p=1$ has been already obtained in Proposition \ref{Cuculescu}, hence, we only show the result for $1<p<\infty$.  	Note that, for every projection $e$,  $ \mu_t(e)=\mathds{1}_{[0,\tau(e)}(t)$, $t\geq0$.   Applying \cite[Theorem 4.2]{FK1986}, we get
\begin{align*}
\tau((\mathbf{1}-q_N^{(\lambda)})|x_N|)&\leq \int_0^1 \mu_t(\mathbf{1}-q_N^{(\lambda)})\mu_t(|x_N|)dt\\
&\leq \left(\int_0^1 \mu_t(|x_N|)^pdt\right)^{1/p} \left(\int_0^1 \mu_t(\mathbf{1}-q_N^{(\lambda)})^{p'}dt\right)^{1/{p^{\prime}}}\\
&=\|x_N\|_p  [\tau(\mathbf{1}-q_N^{(\lambda)})]^{1/{p'}},
\end{align*}
where $p'$ is the conjugate index of $p$. Combining Proposition \ref{Cuculescu}(iv), we have
	$$\lambda\T ( {\bf 1}-q_N^{(\lambda)}  )
	\leq    \Big(\tau(\mathbf{1}-q_N^{(\lambda)})\Big)^{\frac 1{p'}}\|x_N\|_{p },$$
	which implies the desired result.
\end{proof}

We now prove Proposition \ref{Nc-Azuma-max} with full details.
\begin{proof}[Proof of Proposition \ref{Nc-Azuma-max}]
For every $\lambda>0$, let $(q_{n}^{(\lambda)})_{n= 1}^N$ be the Cuculescu projections constructed as in Proposition \ref{Cuculescu} associated with the self-adjoint  martingale $x=(x_{n})_{n=1}^N$. Then, using Lemma \ref{lem-e1}, for each $p\geq 1$,
\begin{equation}\label{esti 4}
\tau \left(\mathbf{1}-q_N^{(\lambda)}\right)\leq \frac{\|x_N\|_{p}^{p}}{\lambda^{p}}.
\end{equation}
By scaling we assume without loss of generality that $\sum_{j=1}^{N}c^{2}_{j}\leq 1$. By \eqref{eq:d-r} and Theorem \ref{Nc-Azuma}, we obtain
\begin{equation}\label{esti 5}
\begin{aligned}
\|x_N\|_p^{p}&=p\int_0^{\infty}\lambda^{p-1} \tau[\mathds{1}_{(\lambda,\infty)}(|x|)] d\lambda\\
&\leq 2p\int_0^{\infty}\lambda^{p-1}e^{-\lambda^{2}/2}d\lambda \\
&\leq K^{p}p^{p/2},
\end{aligned}
\end{equation}	
for some universal constant $K>0$. Combining with \eqref{esti 4} and \eqref{esti 5} together, we have
\[
\tau \left(1-q_N^{(r)}\right)\leq\inf\limits_{p\geq 1}\frac{\|x\|_{p}^{p}}{\lambda^{p}}\leq\inf\limits_{p\geq 1}\frac{K^{p}p^{p/2}}{\lambda^{p}}=\inf\limits_{p\geq 1}\exp\left\{p\log\left\{\frac{K\sqrt{p}}{\lambda}\right\}\right\}.
\]
	If $\lambda\geq eK$, then, choosing $p=\left(\frac{\lambda}{eKS_N}\right)^{2}\geq 1$, we have,
	\begin{equation*}
		\tau \left(1-q_N^{(\lambda)}\right)\leq \exp\left\{-\frac{\lambda^{2}}{e^{2}K^{2}}\right\}.
	\end{equation*}
 Choosing $D=2e^{2}K^{2}$ we obtain the desired inequality.
\end{proof}

We conclude this subsection with the following well-known result whose proof is completely analogous to the classical setting (see \cite{Ve2018}).

\begin{prop}\label{Lpsi random variables}
Suppose that $(\M,\tau)$ is a tracial von Neumann algebra and $\alpha\in[1,\infty)$. The following statements are equivalent.
\begin{enumerate}[{\rm (i)}]
\item There exists $K>0$ such that $\|x\|_{p}\leq Kp^{1/\alpha}$ for all $p\geq 1$.
\item There exists $c>0$ such that $\tau\left(e^{c|x|^{\alpha}}\right)<\infty$.
\item There exists $d>0$ such that $\tau\left(\mathds{1}_{(r,\infty)}(|x|)\right)\leq e^{-dr^{\alpha}}$ for every $r>0$.
\end{enumerate}
\end{prop}

\begin{proof}
``(i) $\Rightarrow$ (ii)'' follows from the Taylor expansion and the Stirling formula $k!\asymp\frac{k^{k}}{e^{k}}\sqrt{2\pi k}$ (as $k\to \infty$). ``(ii) $\Rightarrow$ (iii)'' follows from the Chebyshev inequality. ``(iii) $\Rightarrow$ (i)'' follows from \eqref{eq:d-r}.
\end{proof}

\section{Large deviation inequalities for sums of noncommutative independent random variables}\label{sec:ldi}

This section is devoted to extending \eqref{eq:ldi} to the noncommutative setting, which characterizes the exponential integrability of noncommutative independent sequences via deviation inequalities. We first recall the famous Golden-Thompson inequality, which serves as one of the fundamental tools in establishing noncommutative deviation inequalities.

\begin{thm}[Golden-Thompson]\label{Golden-Thompson inequality}
For self-adjoint elements $x$ and $y$ in $L_0(\mathcal{M})$, we have
\begin{equation*}
\tau\left(e^{x+y}\right)\leq \tau\left(e^{x} e^{y}\right). 
\end{equation*}
\end{thm}

We now turn to the \emph{Hermitian dilation argument} which enable one to reduce problems to the self-adjoint cases. Consider the algebra $(\mathbb{M}_{2\times 2}, \bar{\mathrm{tr}})$, where $\bar{\mathrm{tr}}$ is the normalized trace. 
For $x\in L_0(\mathcal{M})$ with $\tau(x)=0$, define a mapping $\mathcal{J}:\mathcal{M}\to \mathcal{M}\otimes\mathbb{M}_{2\times 2}$ by setting
\begin{equation}\label{eq:def J}
\mathcal{J}(x)=
\begin{pmatrix}
0&x\\
x^{*}&0
\end{pmatrix}.
\end{equation}
It is clear that $\mathcal{J}(x)$ is self-adjoint. The following two auxiliary lemmas will be used in the Hermitian dilation argument.

\begin{lem}\label{lem:dist}
For each $x\in L_0(\mathcal{M})$, we have 
$$	\tau\otimes\bar{\mathrm{tr}}\left[\mathds{1}_{[r,\infty)}(|\mathcal{J}(x)|)\right]=\tau\left[\mathds{1}_{[r,\infty)}(|x|)\right].$$
\end{lem}

\begin{proof}
Note that 
\begin{equation*}
	\mathds{1}_{[r,\infty)}(|\mathcal{J}(x)|)=
	\begin{pmatrix}
		\mathds{1}_{[r,\infty)}(|x^*|)&0\\
		0&\mathds{1}_{[r,\infty)}(|x|)
	\end{pmatrix}.
\end{equation*}
Also recall that it was shown in \cite[Lemma 2.5(ii)]{FK1986} that for each $x\in L_0(\mathcal{M})$,
\[
\tau\left[\mathds{1}_{[r,\infty)}(|x^{*}|)\right]=\tau\left[\mathds{1}_{[r,\infty)}(|x|)\right].
\]
Then we get
\begin{align*}
	\tau\otimes\bar{\mathrm{tr}}\left[\mathds{1}_{[r,\infty)}(|\mathcal{J}(x)|)\right]&=\frac{1}{2}\left(\tau\left[\mathds{1}_{[r,\infty)}(|x^{*}|)\right]+\tau\left[\mathds{1}_{[r,\infty)}(|x|)\right]\right)\\
	&=\tau\left[\mathds{1}_{[r,\infty)}(|x|)\right].
\end{align*}
\end{proof}

Immediately, from the above lemma, we have 
$$\mu(t,\mathcal{J}(x))=\mu(t,x),\quad \forall t>0.$$
Hence, by Lemma 2.5 (i) in \cite{FK1986}, we have
\begin{equation}\label{eq:infty}
\|\mathcal{J}(x)\|_{\mathcal{M}\bar{\otimes}\mathbb{M}_{2\times2}}=\lim_{t\to0}\mu(t,\mathcal{J}(x))=\lim_{t\to0}\mu(t,x)=\|x\|_{\mathcal{M}}.
\end{equation}

\begin{lem}\label{lem:ind and mart}
Let $(d_{j})_{j\geq 1}\subseteq\bigcap_{p\geq 1}L_{p}(\M)$, then we have following claim.
\begin{enumerate}[{\rm (i)}]
\item If $(d_{j})_{j\geq 1}$ are mean zero and independent (with respect to $\tau$), then $(\mathcal{J}(d_{j}))_{j\geq1}$ are self-adjoint and independent (with respect to $\tau\otimes \bar{\mathrm{tr}}$).
\item If  $ (d_{j})_{j\geq1}$ are martingale differences (with respect to $(\mathcal{M}_{j})_{j\geq1}$), then $(\mathcal{J}(d_{j}))_{j\geq1}$ are self-adjoint martingale differences (with respect to $(\mathcal{M}_{j}\bar{\otimes}\mathbb{M}_{2\times2})_{j\geq1}$).
\end{enumerate}
\end{lem}

\begin{proof}
 (i) It is easy to see  $\mathcal{J}(d_{j})$ is mean zero for each $j\geq1$. To show $(\mathcal{J}(d_{j}))_{j\geq1}$ is independent, it suffices to check that for any $j,k\geq1$ with $j\neq k$ the following holds
\[
\tau \otimes \bar{\mathrm{tr}} [\mathcal{J}(d_k)^m \mathcal{J}(d_j)^n]=\tau \otimes \bar{\mathrm{tr}} [\mathcal{J}(d_k)^m ] \cdot \tau \otimes \bar{\mathrm{tr}} [\mathcal{J}(d_j)^n],\quad\forall~m,n\in \mathbb{N}.
\]
If one of $\{m,n\}$ is odd, then we can see that both side of the above equality are equal to zero. Hence, it remains to deal with the case when both $m,n$ are even. In this case, we have 
\begin{equation*}
\mathcal{J}(d_k)^m=
\begin{pmatrix}
|d_k^*|^m&0\\
0&|d_k|^m
\end{pmatrix},\quad \mathcal{J}(d_j)^n=
\begin{pmatrix}
|d_k^*|^n&0\\
0&|d_k|^n
\end{pmatrix}.
\end{equation*}
Then 
\begin{align*}
\tau \otimes \bar{\mathrm{tr}} [\mathcal{J}(d_k)^m \mathcal{J}(d_j)^n]&= \frac{1}{2}[\tau(|d_k^*|^m|d_j^*|^n)+ \tau(|d_k|^m|d_j|^n)]\\
&=\frac{1}{2}[\tau(|d_k^*|^m)\tau(|d_j^*|^n)+ \tau(|d_k|^m)\tau(|d_j|^n)]\\
&=\frac{1}{2}[\tau(|d_k|^m)\tau(|d_j|^n)+ \tau(|d_k|^m)\tau(|d_j|^n)]\\
&=\tau(|d_k|^m)\tau(|d_j|^n)=\tau \otimes \bar{\mathrm{tr}} [\mathcal{J}(d_k)^m ] \cdot \tau \otimes \bar{\mathrm{tr}} [\mathcal{J}(d_j)^n],
\end{align*}
where we used the independence of $d_{j}$ and $d_{k}$ for $j\neq k$ are independent in the second equality. The proof of (ii) is complete analogous to (i), and we omit the detail.
\end{proof}

By Lemma \ref{lem:dist} and Lemma \ref{lem:ind and mart}, we can reduce  noncommutative deviation inequalities to their corresponding self-adjoint counterparts. In particular, the self-adjointness assumption in Theorem \ref{Nc-Azuma} can be dropped. Indeed, for a general noncommutative martingale $(x_k)_{k\geq1}$, combining Lemma \ref{lem:dist}, Lemma \ref{lem:ind and mart}, \eqref{eq:infty} and Theorem \ref{Nc-Azuma},  we have
\begin{align*}
\tau(\mathds{1}_{(r, \infty)}(|x_{n}|))&=\tau\otimes\bar{\mathrm{tr}}\left[\mathds{1}_{[r,\infty)}(|\mathcal{J}(x_n)|)\right]\leq 2\exp\left\{\frac{-r^{2}}{2\sum_{j=1}^{n}c_j^2}\right\}.
\end{align*} 

The main result of this section is the following characterization of exponential integrability of noncommutative independent sequences in terms of deviation inequalities.

\begin{thm}\label{thm:ldi}
Suppose that $(d_{j})_{j=1}^{\infty}$ is a sequence of  independent mean zero random variables and let $S_{n}=\sum_{j=1}^{n}d_{j}$ for all $n\in \N$. Then the following statements are equivalent:
\begin{enumerate}[{\rm (i)}]
\item There is a universal constant $c>0$ such that for each $n\in \N$ and $r>0$ we have
\[
\tau\left[\mathds{1}_{(nr,\infty)}(|S_{n}|)\right]\leq 4\exp\left\{-cnr\right\}.
\]
\item The uniform exponential integrability of the sequence $(x_{j})_{j\in \mathbb{N}}$, that is,
$\sup_{j\in \N}\tau(e^{|x_{j}|})<\infty$.
\end{enumerate}
\end{thm}

\begin{proof}
To show (i) implies (ii), it suffices to assume that $c=1$. For every $n\in \N$, by the assumption, we have
\[
\left\|\frac{1}{n}S_{n}\right\|^{p}_{p}=\int_{0}^{\infty}~p r^{p-1}\tau\left[\mathds{1}_{(r,\infty)}\left(\frac{1}{n}|S_{n}|\right)\right]dr\leq 4\int_{0}^{\infty}~pr^{p-1}e^{-rn}dr.
\]
Changing the variable $rn$ to $\lambda$ we have
\begin{equation}\label{exp int 1}
\|S_{n}\|^{p}_{p}\leq\int_{0}^{\infty}p\lambda^{p-1}e^{-\lambda}d\lambda\leq K^{p}_{1}p^{p},~\mbox{for}~p\geq 1,~\mbox{and}~n\in \N.
\end{equation}
By the triangle inequality and \eqref{exp int 1}, we get
\begin{equation}\label{exp int 2}
\|d_{n}\|^{p}_{p}\leq (\|S_{n}\|_{p}+\|S_{n-1}\|_{p})^{p}\leq 2^{p-1}(\|S_{n}\|^{p}_{p}+\|S_{n-1}\|^{p}_{p})\leq (2K_{1})^{p}p^{p},~\mbox{for all}~p\geq 1.
\end{equation}
Hence, by Proposition \ref{Lpsi random variables}, there exists $K_{3}>0$ with $\tau\left(e^{|d_{n}|}\right)\leq K_{3}$ for every $n\in \N$.

Conversely, assume that $(d_{j})_{j=1}^{\infty}$ are self-adjoint with $\|d_{j}\|^{p}_{p}\leq p^{p}$ for all $p\geq 1$ and $j\in \N$. For each $\lambda>0$, the Golden-Thompson inequality Theorem \ref{Golden-Thompson inequality} implies that
\[
\tau\left(e^{\lambda S_{n}}\right)\leq \tau\left(e^{\lambda S_{n-1}}e^{\lambda d_{n}}\right)= \tau\left(e^{\lambda S_{n-1}}\right)\tau\left(e^{\lambda d_{n}}\right),
\]
where we used the fact $(d_j)_{j=1}^{\infty}$ is independent.
Applying the Taylor expansion to $\tau(e^{\lambda d_{n}})$ and noting $\tau(d_{n})=0$, we get
\[
\tau(e^{\lambda d_{n}})=1+\sum_{p\geq 2}\frac{\lambda^{p}\|d_{n}\|^{p}_{p}}{p!}\leq 1+\sum_{p\geq 2}(e\lambda)^{p}=1+\frac{e^{2}\lambda^{2}}{1-e\lambda},~\mbox{whenever}~|e\lambda|<1,
\]
where we use the Stirling approximation $p!\geq (p/e)^{p}$ in the first inequality. Moreover, when $|e\lambda|<1/2$ we can further estimate $1+\sum_{p\geq 2}\frac{e^{2}\lambda^{2}}{1-e\lambda}$ as  follows
\[
1+\frac{e^{2}\lambda^{2}}{1-e\lambda}\leq 1+2e^{2}\lambda^{2}\leq \exp\{e^{2}\lambda^{2}\}.
\]
Iterating the argument and applying the Chernoff bound, we have
\[
\tau\left(\mathds{1}_{(nr,\infty)}(S_{n})\right)\leq \inf_{0<\lambda<1/2e}\exp\{-\lambda rn+\lambda^{2}e^{2}n\},\mbox{for each}~n\in\N,~r>0.
\]
For $r>e/(e+4)$, we choose $\lambda=\frac{1}{4e}$, and it is easy to see that $\exp\{-\frac{rn}{4e}+\frac{n}{16}\}\leq \exp\{-\frac{rn}{16}\}$. Hence
\begin{equation}\label{exponential 1}
\tau\left(\mathds{1}_{(nr,\infty)}(S_{n})\right)\leq \exp\{-rn/16\},~n\in \N,~r>e/(e+4).
\end{equation}
For $0<r\leq e/(e+4)$, we choose $\lambda=\frac{r}{2e^{2}}$, then the fact $\frac{r(1-r)n}{4e^{2}}>-\log 2$ yields that
\begin{equation}\label{exponential 2}
\tau\left(\mathds{1}_{(nr,\infty)}(S_{n})\right)\leq 2\exp\left\{-\frac{rn}{4e^{2}}\right\},~n\in \N,~0\leq r\leq e/(e+4).
\end{equation}
Combing \eqref{exponential 1} and \eqref{exponential 2} yields that
\[
\tau\left(\mathds{1}_{(nr,\infty)}(S_{n})\right)\leq 2\exp\{-rn/16\},~n\in \N,~r>0.
\]
Since the sequence $(x_{j})_{j=1}^{\infty}$ is self-adjoint, we obtain via functional calculus that
\begin{equation}\label{eq:self}
\tau\left(\mathds{1}_{(nr,\infty)}(|S_{n}|)\right)\leq 4\exp\{-nr/16\},~n\in\N,~r>0.
\end{equation}

Now we consider general independent and mean zero $(d_j)_{j=1}^\infty$. Denote $\widetilde{d}_j:=\mathcal{J}(d_j)$, where $\mathcal{J}$ is as in \eqref{eq:def J}. Then, using Lemma \ref{lem:dist} and \cite[Corollary 2.8]{FK1986}, we know that, for each $j$,
$$\tau\otimes \bar{\mathrm{tr}} (e^{|\widetilde{d}_j|})=\tau (e^{|d_j|}),$$
which means $\sup_j\tau\otimes \bar{\mathrm{tr}} (e^{|\widetilde{d}_j|})<\infty$. According to Lemma \ref{lem:ind and mart}, we see that $(\widetilde{d}_j)_{j=1}^{\infty}$ is a sequence of self-adjoint independent and mean zero random variables. Due to Lemma \ref{lem:dist} and \eqref{eq:self},  we have
\[
\tau(\mathds{1}_{(r, \infty)}(|S_{n}|))=\tau\otimes\bar{\mathrm{tr}}\left[\mathds{1}_{[r,\infty)}(|\mathcal{J}(S_n)|)\right]\leq 4\exp\{-nr/16\},~n\in\N,~r>0,
\]
which completes the proof.
\end{proof}

\section{Large deviation inequalities for noncommutative martingales}\label{sec:ldi md}
In this section, we provide large deviation inequalities for noncommutative martingales under (modified) Cram\'{e}r condition and $L_p$-boundedness condition. Specifically, main results are Theorem \ref{deviation 1}, Theorem \ref{thm:ldi mcc}, and Theorem \ref{Lp-version 1}, which are noncommutative extensions of \eqref{eq:ldi md}-\eqref{eq:ldi mcd}. We conclude this section with a summary on optimality and comments of our results. 
	
\subsection{Deviation inequalities under (modified) Cram\'{e}r condition}
\begin{thm}\label{deviation 1}
Let $(x_{k})_{k\geq 1}$ be a noncommutative martingale such that $\sup_{k\geq1}\tau\left(e^{|dx_{k}|}\right)<\infty$. Then, for each $n\in \mathbb{N}$, $r>0$ and $\varepsilon\in (0,1)$, we have	
\begin{equation*}
\tau\left(\mathds{1}_{(nr,\infty)}(|x_{n}|)\right)\leq 6\exp\left\{-\frac{(1-\varepsilon) r^{2/3}n^{1/3}}{2}\right\}.
\end{equation*}
\end{thm} 
Moreover, we establish the following large deviation inequality which includes Theorem \ref{deviation 1} as a special case. 
\begin{thm}\label{thm:ldi mcc}
Let $(x_{k})_{k\geq 1}$ be a noncommutative martingale such that
\[
\sup_{k\geq1}\tau[\exp(|dx_{k}|^{\frac{2\alpha}{1-\alpha}})]<\infty,\quad\mbox{for some } \alpha\in(0,1).
\]
Then, there exists $c_{\alpha,r} >0$ such that
\begin{equation*}
\tau\left(\mathds{1}_{(nr,\infty)}(|x_{n}|)\right)\leq c_{\alpha,r}\exp\left\{-r^{2\alpha}n^{\alpha}/16^{\alpha}\right\},~\mbox{for}~r>0,~n\in \N.
\end{equation*}	
\end{thm}	
	
\begin{proof}
We  use a truncating argument as in \cite{LV2001} (see also \cite{FGL2012}) to prove the result. 
For a fixed $\alpha>0$ we truncate the martingale $(x_{k})_{k=0}^{\infty}$ with parameter $u>0$ as follows:
\begin{equation*}
dy_{k}\coloneqq dx_{k}\mathds{1}_{[0,u]}(|dx_{k}|)-\Ex_{k-1}\left[dx_{k}\mathds{1}_{[0,u]}(|dx_{k}|)\right],
\end{equation*}
and
\begin{equation*}
dz_{k}\coloneqq dx_{k}\mathds{1}_{(u,\infty)}(|dx_{k}|)-\Ex_{k-1}\left[dx_{k}\mathds{1}_{(u,\infty)}(|dx_{k}|)\right],
\end{equation*}
		Denote
		\begin{equation*}
			y_{n}\coloneqq\sum_{k=1}^{n}dy_{k}\quad\mbox{and}\quad z_{n}\coloneqq \sum\limits_{k=1}^{n}dz_{k}.
		\end{equation*}
It is obvious that $x_{n}=y_{n}+z_{n}$ for each $n$, and both $(y_{n})_{n\geq 1}$ and $(z_{n})_{n\geq 1}$ are martingales. For arbitrary $t\in(0,1)$, it follows from \cite[Lemma 2.1]{JRWZ2020} that
		\begin{equation}\label{estimate 1}
			\tau\left(\mathds{1}_{(r,\infty)}(|x_{n}|)\right)\leq \tau\left(\mathds{1}_{(rt,\infty)}(|y_{n}|)\right)+\tau\left(\mathds{1}_{(r(1-t),\infty)}(|z_{n}|)\right).
		\end{equation}
		We shall estimate the two terms in the right hand side of \eqref{estimate 1} separately. 
		
Combining the noncommutative Azuma inequality (i.e., Theorem \ref{Nc-Azuma}) with Lemma \ref{lem:dist} and Lemma \ref{lem:ind and mart} we get that
		\begin{equation}\label{estimate 6}
				\tau\left(\mathds{1}_{(rt,\infty)}(|y_{n}|)\right)\leq 2\exp\left\{-\frac{r^{2}t^{2}}{8nu^2}\right\},
		\end{equation}
		where we used the fact $\left\|dy_{k}\right\|_{\M}\leq 2u$ for each $k\in \N$.
		
To bound the  term $\tau\left(\mathds{1}_{(r(1-t),\infty)}(|z_{n}|)\right)$, we apply Chebyshev inequality to get
\begin{equation}\label{eq:zn}
	\tau\left(\mathds{1}_{(r(1-t),\infty)}(|z_{n}|)\right)\leq \frac{\|z_n\|_2^2}{r^2(1-t)^2}=\sum_{k=1}^n\frac{\|dz_k\|_2^2}{r^2(1-t)^2},
\end{equation}
where the equality is due to the orthogonality of martingale differences.
		For each $1\leq k\leq n$, basic calculation gives us 
		\begin{equation*}
			\begin{aligned}
				\|dz_{k}\|^{2}_{2}&=\tau\left((dx_{k})^{2}\mathds{1}_{(u,\infty)}(|dx_{k}|)\right)-\left\|\Ex_{k}\left[dx_{k}\mathds{1}_{(\alpha n^{1/3},\infty)}(|dx_{k}|)\right]\right\|^{2}_{2}\\
				&\leq\tau\left((dx_{k})^{2}\mathds{1}_{(u,\infty)}(|dx_{k}|)\right),
				\end{aligned}
		\end{equation*}
	which, together with \eqref{eq:dsb}, implies 
	\begin{equation}\label{estimate 2}
				\|dz_{k}\|^{2}_{2}\leq -\int_u^{\infty} t^2 dF_k(t),
	\end{equation}
where $F_k(t)=\tau(\mathds{1}_{(u,\infty)}(|dx_k|))$. It follows from the Chebyshev inequality and  the assumption $\sup_{k\geq1}\tau[\exp(|dx_{k}|^{\frac{2\alpha}{1-\alpha}})]\leq K$ that for each $k\geq1$ and $t>0$ we have
		\begin{equation}\label{estimate 3}
			F_{k}(t)\leq \exp\{-t^{\frac{2\alpha}{1-\alpha}}\} \tau[\exp(|dx_{k}|^{\frac{2\alpha}{1-\alpha}})] \leq K\exp\{-t^{\frac{2\alpha}{1-\alpha}}\}.
		\end{equation}
		Substituting \eqref{estimate 3} into \eqref{estimate 2} yields
		\begin{align}\label{eq:zk}
			\|dz_{k}\|^{2}_{2}&\leq u^2F_k(u)+ \int_u^{\infty} 2t F_k(t) dt\nonumber\\
			&\leq Ku^2 \exp\{-u^{\frac{2\alpha}{1-\alpha}}\}+ 2K \int_u^{\infty} t\exp\{-t^{\frac{2\alpha}{1-\alpha}}\} dt.
		\end{align}
	Observe that the function $g(t)=t^3\exp\{-t^{\frac{2\alpha}{1-\alpha}}\}$ is decreasing in $[\beta,\infty)$ and is increasing in $[0,\beta]$ with
\[
\beta=\left(\frac{3(1-\alpha)}{2\alpha}\right)^{\frac{1-\alpha}{2\alpha}}.
\]
Then, for $u<\beta$, we have
	\begin{equation}\label{eq:beta 1}
		\begin{aligned}
			\int_{u}^{\infty}t\exp\{-t^{\frac{2\alpha}{1-\alpha}}\}dt&\leq \int_{u}^{\beta}t\exp\{-t^{\frac{2\alpha}{1-\alpha}}\}dt+\int_{\beta}^{\infty}t^{-2}t^3\exp\{-t^{\frac{2\alpha}{1-\alpha}}\}dt\\
			&\leq \int_{u}^{\beta}t\exp\{-u^{\frac{2\alpha}{1-\alpha}}\}dt+\int_{\beta}^{\infty}t^{-2}\beta^3\exp\{-\beta^{\frac{2\alpha}{1-\alpha}}\}dt\\
			&\leq \frac{3}{2}\beta^2\exp\{-u^{\frac{2\alpha}{1-\alpha}}\}.
		\end{aligned}
	\end{equation}
For the case $u\geq\beta$, we can similarly show 
\begin{equation}\label{eq:beta 2}
		\int_{u}^{\infty}t\exp\{-t^{\frac{2\alpha}{1-\alpha}}\}dt\leq u^2\exp\{-u^{\frac{2\alpha}{1-\alpha}}\}.
\end{equation}
Combining \eqref{estimate 2}, \eqref{eq:zk}, \eqref{eq:beta 1}, and \eqref{eq:beta 2} , we have, for each $k$,
$$\|dz_k\|_2^2\leq 3K(u^2+\beta^2) \exp\{-u^{\frac{2\alpha}{1-\alpha}}\}.$$
which, together with \eqref{eq:zn}, gives us
\begin{equation}\label{eq:yn}
		\tau\left(\mathds{1}_{(r(1-t),\infty)}(|z_{n}|)\right)\leq \frac{n}{r^2(1-t)^2} 3K (u^2+\beta^2) \exp\{-u^{\frac{2\alpha}{1-\alpha}}\}.
\end{equation}
	
	Now, we conclude from \eqref{estimate 1}, \eqref{estimate 6} and \eqref{eq:yn} that 
	\begin{equation*}
		\tau\left(\mathds{1}_{(r,\infty)}(|x_{n}|)\right)		\leq 2\exp\left\{-\frac{r^{2}t^{2}}{8nu^2}\right\}+ \frac{n}{r^2(1-t)^2} 3K (u^2+\beta^2) \exp\{-u^{\frac{2\alpha}{1-\alpha}}\}.
	\end{equation*}
Taking $t=1/\sqrt{2}$ and $u=(\frac{r}{4\sqrt{n}})^{1-\alpha}$, we deduce from the above inequality that, for each $r>0$,
	\begin{equation}\label{eq:r}
	\tau\left(\mathds{1}_{(r,\infty)}(|x_{n}|)\right)		\leq c_{\alpha,r,n} \exp\Big\{-\Big(\frac{r^2}{16n}\Big)^{\alpha}\Big\},
\end{equation}
	where 
	$$c_{\alpha,r,n}=2+15nK\Big(\frac{1}{r^{2\alpha}(16n)^{1-\alpha}}+\frac{\beta^2}{r^2}\Big)\leq 2+15K\Big(\frac{n^{2\alpha}}{r^{2\alpha}16^{1-\alpha}}+\frac{n^2\beta^2}{r^2}\Big).$$
	The desired inequality of the theorem follows by replacing $r$ by $nr$ in \eqref{eq:r}.
\end{proof}

\subsection{Deviation inequality under $L_p$-boundedness condition}
We aim to derive a deviation inequality for noncommutative martingales fulfill the $L_{p}$-boundedness condition. Before going further, we apply the noncommutative Burkholder-Gundy inequality to obtain the following elementary lemma. 

\begin{lem}\label{Lp bounded lemma}
Let $1<p<\infty$, and let  $(x_k)_{k\geq1}$ be a noncommutative $L_{p}$-martingale. Suppose that $\sup_k\|dx_k\|_p\leq K$ for some $K>0$.  Then there exists a positive constant $C_p$ such that
\[
\|x_n\|_p\leq C_p n^{\frac{\max\{2,p\}}{2}}K^p,~\mbox{for each}~n\in \N.
\]
\end{lem}

\begin{proof}
Note that $\|\cdot\|_{L_r(\mathcal{M})}$ is  a $r$-norm with $0<r<1$. For $1<p<2$, using Theorem \ref{Nc-BG}, we immediately have
\begin{align*}
\|x_n\|_p^p&\leq A_p \left\|\left(\sum_{k=1}^n|dx_k|^2\right)^{1/2}\right\|_p^p=\left\|\sum_{k=1}^n|dx_k|^2\right\|_{p/2}^{p/2}\\
&\leq A_p\sum_{k=1}^n\||dx_k|^2\|_{p/2}^{p/2}=\sum_{k=1}^n\|dx_k\|_{p}^{p}\leq A_p nK^p.
\end{align*}
For   $2\leq p<\infty$, we apply Theorem \ref{Nc-BG} again and the triangle inequality to get
\begin{align*}
	\|x_n\|_p&\leq B_p\max\left\{\left\|\left(\sum_{k=1}^n|dx_k|^2\right)^{1/2}\right\|_p,\left\|\left(\sum_{k=1}^n|dx_k^*|^2\right)^{1/2}\right\|_p \right\}\\
	&=B_p\max\left\{\left\|\sum_{k=1}^n|dx_k|^2\right\|_{p/2}^{1/2},\left\|\sum_{k=1}^n|dx_k^*|^2\right\|_{p/2}^{1/2}\right\}\\
	&\leq 2B_p\left(\sum_{k=1}^n\|dx_k\|_{p}^2\right)^{1/2}\leq 2B_p n^{1/2}K.
\end{align*}
The desired inequality follows. 
\end{proof}

\begin{thm}\label{Lp-version 1}
Suppose that $(x_{n})_{n=0}^{\infty}$ is a $L_{p}$-bounded martingale for $1\leq p<\infty$ and $M>0$ such that $\|dx_{k}\|_{L_{p}(\M)}\leq M$ for every $k\in \N$. Then there exists a constant $C_{p}>0$ denpending only on $p$ such that
\begin{equation}
\tau\left(\mathds{1}_{(nr,\infty)}(|x_{n}|)\right)\leq \frac{C_{p}M^{p}}{r^{p}n^{p\left(1-\frac{1}{\min\{p,2\}}\right)}},~\mbox{for}~r>0.
\end{equation}
\end{thm}
	
\begin{proof}
Note that for $p=1$ the right hand side of the desired inequality does not capture any information of $n$, and hence it is trivial for this case. And for the case $1<p<\infty$,  we get, for each $n\in \mathbb{N}$,
\begin{equation}\label{Chebyshev inequality for Lp}
\tau\left(\mathds{1}_{(nr,\infty)}(|x_{n}|)\right)\leq \frac{\|x_n\|_p^p}{n^pr^p}.
\end{equation}
Substituting the estimate obtained via Lemma \ref{Lp bounded lemma} to \eqref{Chebyshev inequality for Lp} yields the desired inequality.
\end{proof}
Remarks on optimality and refinements are summarized as follows. 
\begin{com}
Theorem \ref{deviation 1}-\ref{Lp-version 1} are optimal, because the optimality have been evidenced in the commutative setting and we refer to \cite{FGL2012} and \cite{LV2001} for more details.
\end{com}
 
\begin{rem}
We will obtain a maximal version of the inequality for the exponential case when applying the Cuculescu projections as we have done in our maximal version of noncommutative Azuma inequality. For the $L_{p}$ case, a direct strengthening version of the Bukholder-Gundy inequality yields the desired strengthening.
\end{rem}

\section{Application in noncommutative egodic theory}\label{sec:application}

In this section, we apply the large deviation inequalities for noncommutative martingales established in Sect. \ref{sec:ldi md} to the study of noncommutative ergodic theory. The approach presented below is inspired by method derived by Lesigne and Voln\'{y} \cite{LV2001}. Assume from now on that $T:\mathcal{M}\to \mathcal{M}$ is a linear mapping fulfilling the following conditions:
\begin{enumerate}[(i)]
\item $T:\mathcal{M}\to \mathcal{M}$ is a $*$-isomorpshim;
\item $T$ is trace preserving, that is, $\tau\circ T=\tau$;
\item $T$ extends to be a bounded self-adjoint operator on $L_{2}(\mathcal{M})$, that is, $\tau\left(T(y)^{*}x\right)=\tau(y^{*}T(x))$ for every $x,y\in L_2(\mathcal{M})$;
\item $T$ is normal;
\item there exists a invariant von Neumann sub-algebra $\mathcal{A}$ of $\mathcal{M}$, that is, $\mathcal{A}\subseteq T^{-1}(\mathcal{A})$.
\end{enumerate}

For each $j\in \Z$,  let $T^{0}$ be the identity mapping and set
\begin{equation}
\mathcal{A}_j\coloneqq T^{-j}(\mathcal{A}).
\end{equation}
Then $(\mathcal{A}_j)_{j\in \mathbb{Z}}$ forms an increasing sequence of von Neumann sub-algebras in $\M$.
	
\begin{rem}
It is worthwhile to point out that $T^{-j}(\mathcal{A})$ is a von Neumann algebra for each $j\in \Z$. Indeed, for negative $j\in \Z$, $T^{-j}(\mathcal{A})$ is a von Neumann algebra follows from the normality of $T$ directly. To verify $T^{-j}(\mathcal{A})$ is a von Neumann for positive $j\in \Z$, by induction argument, it suffices to verify that $T^{-1}(\mathcal{A})$ is a von Neumann algebra. This   follows from the fact that $T^{-1}$ is also $w^{*}$-to-$w^{*}$ continuous (i.e. $T^{-1}$ is normal). By the Banach-Dieudonn\'e theorem, it suffices to prove that the graph $\{(x,T^{-1}(x)):x\in \mathcal{M}\}$ is $w^{*}$-closed. Indeed, for every net $(x_{\alpha},T^{-1}(x_{\alpha}))_{\alpha\in\triangle}\subseteq \{(x,T^{-1}(x):x\in \mathcal{M})\}$ such that $(x_{\alpha},T^{-1}(x_{\alpha}))\to (y,z)$ in the $w^{*}$-topology, we now verify that $z=T^{-1}(y)$. Since $T^{-1}(x_{\alpha})\stackrel{w^{*}}{\to}z$, the $w^{*}$-to-$w^{*}$ continuity (i.e. normality) of $T$ yields that $x_{\alpha}\stackrel{w^{*}}{\to}T(z)$. On the other hand, $x_{\alpha}\stackrel{w^{*}}{\to}y$ by assumption. Hence, we have $y=T(z)$, that is, $z=T^{-1}(y)$.
\end{rem}
	
Let $\mathcal{A}_{\infty}\coloneqq \bigvee_{j\in \mathbb{Z}}\mathcal{A}_{j}$ be the von Neumann algebra generated by $(\mathcal{A}_{j})_{j\in \Z}$ and set $\mathcal{A}_{-\infty}\coloneqq \bigcap_{j\in \Z}\mathcal{A}_{j}$. For each $j\in \N$, let $\mathcal{E}_{j}$ be the conditional expectation from $\M$ onto $\A_{j}$. For $f\in L_{1}(\M)$, we assume without loss of generality that $\mathcal{E}_{-\infty}(f)=\tau(f)=0$ and $\mathcal{E}_{\infty}(f)=f$. For each $j\in \Z$, define the martingale difference operator $d_{j}:L_{1}(\M)\to L_{1}(\M)$ as follows
	\begin{equation}
		d_{j}(f)\coloneqq \mathcal{E}_j\left[f\right]-\mathcal{E}_{j-1}\left[f\right].
	\end{equation}
	It follows from the definition of $d_{j}$ that 
	\begin{equation}\label{nc-martingale decomposition 1}
		f=\sum\limits_{j\in \Z}d_{j}(f).
	\end{equation}

\begin{lem}\label{lemma nc-1}
Keeping notations as above, we have
\begin{enumerate}[{\rm (i)}]
\item  $\mathcal{E}_{j-1}d_j=0$; $d_{i}d_{j}=0$, if $i\not=j$;
\item $Td_{j}=d_{j+1}T$;
\item $d_{j}T^{-1}=T^{-1}d_{j+1}$.
\end{enumerate}
\end{lem}
	
\begin{proof}
Item (i) is trivial. For $j\in \mathbb{Z}$ and projection $p\in \A_{j+1}$ the $\ast$-isomorphism of $T$ implies that $T(p)$ is a projection in $\A_j$.  Then, by Assumption (iv) on $T$, we have 
\begin{align*}
\tau(\mathcal{E}_{j+1}(Tf) p)&=\tau(Tf \cdot p)=\tau(f\cdot Tp)\\
&=\tau(\mathcal{E}_j(f)\cdot Tp)=\tau(T\mathcal{E}_j(f)p).
\end{align*}
Since $p$ is arbitrary, it follows that 
\begin{equation}\label{ET}
\mathcal{E}_{j+1}(Tf) =T\mathcal{E}_{j}(f),
\end{equation}
which gives us item (ii) because $T$ is linear. Item (iii) follows from item (ii) directly.
\end{proof}
	
%
	
The following proposition is an immediate consequence of the martingale convergence theorem.
\begin{prop}\label{con-L1}
Assume that $f\in L_1(\M)$ such that $\mathcal{E}_{-\infty}(f)=\tau(f)=0$.
\begin{enumerate}[{\rm (i)}]
\item $\mathcal{E}_j(f)\to f$ as $j\to\infty$ in $L_1(\M)$; $\mathcal{E}_j(f)\to 0$ as $j\to-\infty$ in $L_1(\M)$.
\item $\V_{j}(f)=f-\mathcal{E}_{-j}(f)\to f$ as $j\to\infty$ in $L_1(\M)$;  $\V_{j}(f)=f-\mathcal{E}_{-j}(f)\to 0$ as $j\to-\infty$ in $L_1(\M)$. 
\end{enumerate}
\end{prop}

	This proposition together with \eqref{ET} implies
	\begin{equation}
		\mathcal{E}_{0}T^{k}(f)=T^{k}\mathcal{E}_{-k}(f)\to 0,
	\end{equation}
	and
	\begin{equation}
		\V_{0}T^{-k}(f)=T^{-k}\V_{k}(f)\to 0
	\end{equation}
	in $L_{1}$-norm as $k\to \infty$.

	\begin{thm}\label{e-decomposition}
	For $1\leq p<\infty$, and  $f\in L_{p}(\M)$ with $\tau(f)=0$, the following statements are equivalent.
		\begin{enumerate}[{\rm (i)}]
			\item There exist  $m,~g\in L_p(\M)$ such that $f=m+g-Tg$ and $(T^{j}m)_{j\in \Z}$ forms  a sequence of martingale differences with respect to  $(\mathcal{A}_{j})_{j\in \Z}$;
			\item Both $\sum_{k=0}^{\infty}\mathcal{E}_0(T^{k}f)$ and $\sum_{k=0}^{\infty}\left[T^{-k}(f)-\mathcal{E}_0(T^{-k}(f))\right]$
			converge in $L_{p}$-norm.
		\end{enumerate}
	\end{thm}
	
\begin{proof}
It suffices to provide the proof for the case $p=1$ since the same proof still works well for the general case. 
		
		The ``$(i)\Leftarrow (ii)$'' part. Assuming (ii) holds, by Proposition \ref{con-L1} and \eqref{ET}, it is easy to see that, for each $j\in \Z$,
		\begin{equation}\label{nc-form 1}
			\sum\limits_{k=0}^{\infty}\mathcal{E}_{j} T^{k}(f)=T^{j}\left(\sum\limits_{k=0}^{\infty}\mathcal{E}_{0}T^{k-j}(f)\right)=T^{j}\left(\sum\limits_{k=-j}^{\infty}\mathcal{E}_0(T^{k}(f))\right)
		\end{equation}
		converges in $L_{1}(\M)$.
		Similarly,  for each $j\in \Z$, the following sequence 	converges with respect to the $L_{1}$-norm
		\begin{equation}\label{form 2}
			\sum\limits_{k=0}^{\infty}\V_{j}T^{-k}(f)=\sum\limits_{k=0}^{\infty}T^{-k}(f)-\mathcal{E}_{j}T^{-k}(f)=T^{j}\left(\sum\limits_{k=-j}^{\infty}T^{-k}(f)-\mathcal{E}_0\left[T^{-k}(f)\right]\right).
		\end{equation}

		For each $j\in \Z$, define
		\begin{equation}
			g_{j}=
			\begin{cases}
				\sum\limits_{k=0}^{\infty}d_{j}T^{k}(f),~\mbox{if}~j\leq -2,\\
				-\sum\limits_{k=1}^{\infty}d_{j}T^{-k}(f),~\mbox{if}~j\geq -1.
			\end{cases}
		\end{equation}
		Note that
		\begin{equation}
			\begin{split}
				\sum\limits_{j=2}^{n}\sum\limits_{k=0}^{\infty}\mathcal{E}_{-j}T^{k}(f)
				&=\sum\limits_{k=0}^{\infty}\left(\sum\limits_{j=2}^{n}\mathcal{E}_{-j}\right)T^{k}(f)\\
				&=\sum\limits_{k=0}^{\infty}\left(\mathcal{E}_{-2}-\mathcal{E}_{-n-1}\right)T^{k}(f)\\
				&=\sum\limits_{k=0}^{\infty}\mathcal{E}_{-2}T^{k}(f)-\sum\limits_{k=0}^{\infty}\mathcal{E}_{-n-1}T^{k}(f).
			\end{split}
		\end{equation}
		Then, by \eqref{nc-form 1}, we obtain the  convergence of $\sum_{j=2}^{\infty}g_{-j}$ in $L_1(\M)$.
		Analogously, $\sum_{j=-1}^{\infty}g_{j}$ converges in $L_{1}$-norm, which further yields that $g=\sum\limits_{j\in \Z}g_{j}$ is well defined. It follows from Lemma \ref{lemma nc-1} (i) that
$d_{j}(g)=g_{j}$ for all $j\in \Z$. Define
\[
m\coloneqq \sum\limits_{j\in \Z}d_{-1}T^{j}(f).
\]
It is easy to see that $(T^{j}(m))_{j\in\Z}$ forms a sequence of martingale differences with respect to filtration $(\mathcal{A}_{j})_{j\in \Z}$. Indeed, for each $j\in \Z$, it follows from Lemma \ref{lemma nc-1} (i) that
\[
\begin{split}
\mathcal{E}_{j-2}T^{j}(m)&=	\mathcal{E}_{j-2}\left(\sum\limits_{k\in \Z}T^{j}d_{-1}T^{k}(f)\right)=\mathcal{E}_{j-2}\left(\sum\limits_{k\in \Z}d_{j-1}T^{j+k}(f)\right)\\
                                              &=\sum\limits_{k\in \Z}\mathcal{E}_{j-2}d_{j-1}T^{j+k}(f)=0.
\end{split}
\]
We now verify that $f=m+g-Tg$. Since $f=\sum_{j\in \Z}d_{j}(f)$,  it suffices to show that $d_{j}(f)=d_{j}(m+g-Tg)$ for all $j\in \Z$. Indeed, for $j\leq -2$, by Lemma \ref{lemma nc-1} (i), we have
		\begin{equation}\label{lemma case 1}
			d_{j}(m)=d_{j}\left(\sum_{k\in \Z}d_{-1}T^{j}(f)\right)=\sum_{k\in \Z}d_{j}d_{-1}T^{k}(f)=0.
		\end{equation}
		On the other hand, by Lemma \ref{lemma nc-1} (ii), for each $j\leq -2$,
\[
\begin{split}
				d_{j}(g)-d_{j}T(g)&=g_{j}-Td_{j-1}(g)=g_{j}-Tg_{j-1}\\
				&=\sum\limits_{k=0}^{\infty}d_{j}T^{k}(f)-T\left(\sum\limits_{k=0}^{\infty}d_{j-1}T^{k}(f)\right)\\
				&=\sum\limits_{k=0}^{\infty}d_{j}T^{k}(f)-d_{j}\left(\sum\limits_{k=0}^{\infty}T^{k+1}(f)\right)\\
				&=d_{j}\left(\sum\limits_{k=0}^{\infty}T^{k}(f)-\sum\limits_{k=1}^{\infty}T^{k}(f)\right)\\
				&=d_{j}(f).
			\end{split}
\]
		Hence, for each $j\leq -2$, we have $	d_{j}(f)=d_{j}\left(m+g-Tg\right)$. For $j>-1$, by the definition of $m$, we get $d_{j}(m)=0$, and
\[
\begin{split}
				d_{j}\left(m+g-Tg\right)&=g_{j}-d_{j}T(g)=g_{j}-Td_{j-1}(g)=g_{j}-T(g_{j-1})\\
				&=-\sum\limits_{k=1}^{\infty}d_{j}T^{-k}(f)+\sum\limits_{k=1}^{\infty}Td_{j-1}T^{-k}(f)\\
				&=\sum\limits_{k=1}^{\infty}d_{j}T^{-k+1}(f)-\sum\limits_{k=1}^{\infty}d_{j}T^{-k}(f)\\
				&=d_{j}\left(\sum\limits_{k=0}^{\infty}T^{-k}(f)-\sum\limits_{k=1}^{\infty}T^{-k}(f)\right)\\
				&=d_{j}(f).
			\end{split}
\]
		
		For $j=-1$, we have
\[
\begin{split}
d_{-1}\left(m+g-Tg\right)&=d_{-1}(m)+d_{-1}(g)-d_{-1}T(g)\\
&=m+g_{-1}-Td_{-2}(g)\\
&=m+g_{-1}-T(g_{-2})\\
&=\sum\limits_{k\in \Z}d_{-1}T^{k}(f)-\sum\limits_{k=1}^{\infty}d_{-1}T^{-k}(f)-T\left(\sum\limits_{k=0}^{\infty}d_{-2}T^{k}(f)\right)\\
&=\sum\limits_{k\in \Z}d_{-1}T^{k}(f)-\sum\limits_{k=1}^{\infty}d_{-1}T^{-k}(f)-\sum\limits_{k=0}^{\infty}d_{-1}T^{k+1}(f)\\
&=\sum\limits_{k\in \Z}d_{-1}T^{k}(f)-\sum\limits_{k=1}^{\infty}d_{-1}T^{-k}(f)-\sum\limits_{k=1}^{\infty}d_{-1}T^{k}(f)\\
&=d_{-1}(f).
\end{split}
\]
		
We now turn to proof of ``$(i)\Rightarrow (ii)$''. Since $( T^{j}(m))_{j\in \Z}$ is a sequence of martingale differences with respect to the filtration $(\mathcal{A}_{j})_{j\in \Z}$, we assume without loss of generality that $\mathcal{E}_0(m)=m$. To verify the convergence of $\sum_{k=0}^{\infty}\mathcal{E}_0(T^{k}f)$, it suffices to note that $f=m+g-Tg$ and $\lim_{k\to \infty}\|\mathcal{E}_0T^{k}(g)\|_{L_1(\M)}=0$,
		\begin{equation}\label{convergence 1}
			\begin{split}
				\sum_{k=0}^{\infty}\mathcal{E}_0(T^{k}f)
				&=\sum\limits_{k=1}^{\infty}\mathcal{E}_0T^{k}(m+g-Tg)\\
				&=\sum\limits_{k=1}^{\infty}\mathcal{E}_0T^{k}(m)+\mathcal{E}_{0}T(g)-\lim\limits_{k\to \infty}\mathcal{E}_0T^{k+1}(g)\\
				&=\sum\limits_{k=1}^{\infty}\mathcal{E}_0T^{k}(m)+\mathcal{E}_{0}T(g).
			\end{split}
		\end{equation} 
		Since $(T^{j}(m))_{j\in\Z}$ is a sequence of martingale differences with respect to the filtration $(\mathcal{A}_{j})_{j\in \Z}$, it follows that, for each $k\geq 1$, we have
		\begin{equation}\label{convergence 2}
			\mathcal{E}_{0}T^{k}(m)=0.
		\end{equation}
		Substituting \eqref{convergence 2} into \eqref{convergence 1} yields that
		\begin{equation}
			\sum_{k=0}^{\infty}\mathcal{E}_0(T^{k}f)=\mathcal{E}_{0}T(g),
		\end{equation}
		which is just the desired convergence of the series.
		Analogously,
		\begin{equation}\label{convergence 3}
			\begin{split}
				\sum_{k=0}^{\infty}\left[T^{-k}(f)-\mathcal{E}_0(T^{-k}(f))\right]&=\sum\limits_{k=0}^{\infty}\V_{0}T^{-k}(f)=\sum\limits_{k=0}^{\infty}T^{-k}\V_{k}(m+g-Ug)\\
				&=\sum\limits_{k=0}^{\infty}T^{-k}\V_{k}(m)+\V_{0}(g)-\lim\limits_{k\to \infty}T^{-k}\V_{k}(g)\\
				&=\sum\limits_{k=0}^{\infty}T^{-k}\V_{k}(m)+\V_{0}(g).
			\end{split}
		\end{equation}
		
		Note that $\mathcal{E}_{0}(m)=m$, then it follows that $\mathcal{E}_{k}(m)=m$ for each $k\geq 0$. For each $k\geq 1$
		\begin{equation}
			\V_{k}(m)=m-\mathcal{E}_{k}(m)=0,
		\end{equation}
		and, substituting the identity into \eqref{convergence 3} yields the desired convergence of the series. This completes the proof of the theorem.
	\end{proof}
	
	In the sequel, for $f\in L_p(\M)$, let 
	$$S_n(f)=\sum_{k=0}^nT^k(f). $$
	Combining Theorem \ref{e-decomposition} with Theorem \ref{Lp-version 1}, we obtain the following.
	
\begin{cor}
Let $1\leq p<\infty$ and $f\in L_p(\M)$ which satisfies the assumption in Theorem \ref{e-decomposition}. Then we have
\[
\tau\left(\mathds{1}_{(n,\infty)}(|S_{n}(f)|)\right)=O\left(\frac{1}{n^{p\left(1-\frac{1}{\min\{p,2\}}\right)}}\right).
\]
\end{cor}
\begin{proof}
Since $f=m+g-T(g)$, it follows that 
\[
S_{n}(f)=S_{n}(m)+g-T(g),
\]
and consequently, 
\[
\tau\left(\mathds{1}_{(n,\infty)}(|S_{n}(f)|)\right)\leq \tau\left(\mathds{1}_{(n/3,\infty)}(|S_{n}(m)|)\right)+2 \tau\left(\mathds{1}_{(n,\infty)}(|g|)\right).
\]
Note that $g\in L_p(\M)$ implies $\tau\left(\mathds{1}_{(n,\infty)}(|g|)\right)=O(\frac{1}{n^p})$. Hence, the desired result follows from Theorem \ref{Lp-version 1}.
\end{proof}

\noindent{\bf Acknowledgement.}
This work was supported by NSFC (grant Nos. 11961131003, 12001541, 12125109, 12201646 \& 12471134); the Natural Science Foundation Hunan (grant Nos: 2023JJ40696, 2023JJ20058, 2024JJ1010 \& 2024RC3040); the CSU Innovation-Driven Research Programme (grant 2023CXQD016).

%

\providecommand{\bysame}{\leavevmode\hbox to3em{\hrulefill}\thinspace}
\providecommand{\MR}{\relax\ifhmode\unskip\space\fi MR }
\providecommand{\MRhref}[2]{%
  \href{http://www.ams.org/mathscinet-getitem?mr=#1}{#2}
}
\providecommand{\href}[2]{#2}


\end{document}